\documentclass[a4paper]{article}

\usepackage[T1]{fontenc}
\usepackage[latin9]{inputenc}
\usepackage{geometry}
\usepackage{amsmath}
\usepackage{amsthm}
\usepackage{amssymb}
\usepackage{bbm}
\usepackage{color}
\usepackage{paralist}
\usepackage[unicode=true,pdfusetitle,
 bookmarks=true,bookmarksnumbered=false,bookmarksopen=false,
 breaklinks=false,pdfborder={0 0 0},pdfborderstyle={},backref=false,colorlinks=false]
 {hyperref}

\makeatletter

\usepackage[nameinlink,capitalise,noabbrev]{cleveref}

\hypersetup{%
    bookmarksnumbered, bookmarksopen=true, bookmarksopenlevel=1,%
}

\theoremstyle{plain}
\newtheorem{thm}{Theorem}[section]
\crefname{thm}{Theorem}{Theorems}
\theoremstyle{plain}
\newtheorem{lem}[thm]{Lemma}
\crefname{lem}{Lemma}{Lemmas}
\theoremstyle{plain}

\theoremstyle{plain}
\newtheorem*{claim*}{Claim}
\crefname{claim}{Claim}{Claims}
\theoremstyle{definition}

\theoremstyle{plain}

\crefname{conjecture}{Conjecture}{Conjectures}
\theoremstyle{plain}

\crefname{prop}{Proposition}{Propositions}
\theoremstyle{definition}

\theoremstyle{definition}

\theoremstyle{plain}

\crefname{fact}{Fact}{Facts}
\crefformat{equation}{#2(#1)#3}

\crefname{subsection}{Subsection}{Subsections}

\date{}

\crefformat{equation}{#2(#1)#3}

\let\originalleft\left
\let\originalright\right
\renewcommand{\left}{\mathopen{}\mathclose\bgroup\originalleft}
\renewcommand{\right}{\aftergroup\egroup\originalright}
\usepackage{verbatim}

\makeatletter
\renewcommand*{\UrlTildeSpecial}{%
  \do\~{%
    \mbox{%
      \fontfamily{ptm}\selectfont
      \textasciitilde
    }%
  }%
}%
\let\Url@force@Tilde\UrlTildeSpecial
\makeatother

\let\OLDthebibliography\thebibliography
\renewcommand\thebibliography[1]{
  \OLDthebibliography{#1}
  \setlength{\parskip}{0pt}
  \setlength{\itemsep}{3pt plus 0.3ex}
}

\makeatother

\allowdisplaybreaks
\numberwithin{equation}{section}

\begin{document}
\global\long\def\per{\operatorname{per}}%
\global\long\def\RR{\mathbb{R}}%
\global\long\def\FF{\mathbb{F}}%
\global\long\def\QQ{\mathbb{Q}}%
\global\long\def\E{\mathbb{E}}%
\global\long\def\Var{\operatorname{Var}}%
\global\long\def\Cov{\operatorname{Cov}}%
\global\long\def\CC{\mathbb{C}}%
\global\long\def\NN{\mathbb{N}}%
\global\long\def\ZZ{\mathbb{Z}}%
\global\long\def\GG{\mathbb{G}}%
\global\long\def\tallphantom{\vphantom{\sum}}%
\global\long\def\tallerphantom{\vphantom{\int}}%
\global\long\def\supp{\operatorname{supp}}%
\global\long\def\one{\mathbbm{1}}%
\global\long\def\d{\operatorname{d}}%
\global\long\def\Unif{\operatorname{Unif}}%
\global\long\def\Po{\operatorname{Po}}%
\global\long\def\Bin{\operatorname{Binomial}}%
\global\long\def\Ber{\operatorname{Bernoulli}}%
\global\long\def\Geom{\operatorname{Geom}}%
\global\long\def\Rad{\operatorname{Rad}}%
\global\long\def\floor#1{\left\lfloor #1\right\rfloor }%
\global\long\def\ceil#1{\left\lceil #1\right\rceil }%
\global\long\def\falling#1#2{\left(#1\right)_{#2}}%
\global\long\def\cond{\,\middle|\,}%
\global\long\def\su{\subseteq}%
\global\long\def\spn{\operatorname{span}}%
\global\long\def\eps{\varepsilon}%
\global\long\def\one{\boldsymbol{1}}%

\global\long\def\af#1{\textcolor{orange}{\textbf{[AF comments:} #1\textbf{]}}}
\global\long\def\ls#1{\textcolor{blue}{\textbf{[LS comments:} #1\textbf{]}}}
\global\long\def\mk#1{\textcolor{red}{\textbf{[MK comments:} #1\textbf{]}}}

\title{\texorpdfstring{\vspace{-1cm}}{}Singularity of sparse random matrices: simple proofs}
\author{Asaf Ferber\thanks{Department of Mathematics, University of California, Irvine.
Email: \href{mailto:asaff@uci.edu} {\nolinkurl{asaff@uci.edu}}.
Research supported in part by NSF Awards DMS-1954395 and DMS-1953799.}
\and 
Matthew Kwan\thanks{Department of Mathematics, Stanford University, Stanford, CA.
Email: \href{mattkwan@stanford.edu}{\nolinkurl{mattkwan@stanford.edu}}.
Research supported by NSF Award DMS-1953990.}
\and
Lisa Sauermann\thanks{School of Mathematics, Institute for Advanced Study, Princeton, NJ. Email: \href{lsauerma@mit.edu}{\nolinkurl{lsauerma@mit.edu}}. Research supported by NSF Grant CCF-1900460 and NSF Award DMS-2100157.}
}

\maketitle

\begin{abstract}
\noindent
Consider a random $n\times n$ zero-one matrix with ``sparsity'' $p$, sampled according to one of the following two models: either every entry is independently taken to be one with probability $p$ (the ``Bernoulli'' model), or each row is independently uniformly sampled from the set of all length-$n$ zero-one vectors with exactly $pn$ ones (the ``combinatorial'' model). We give simple proofs of the (essentially best-possible) fact that in both models, if $\min(p,1-p)\geq (1+\varepsilon)\log n/n$ for any constant $\varepsilon>0$, then our random matrix is nonsingular with probability $1-o(1)$. In the Bernoulli model this fact was already well-known, but in the combinatorial model this resolves a conjecture of Aigner-Horev and Person.
\end{abstract}

\section{Introduction}

Let $M$ be an $n\times n$ random matrix with i.i.d.\ $\Ber(p)$
entries (meaning that each entry $M_{ij}$ satisfies $\Pr(M_{ij}=1)=p$
and $\Pr(M_{ij}=0)=1-p$). It is a famous theorem of Koml\'os\ \cite{Kom67,Kom68}
that for $p=1/2$ a random Bernoulli matrix is \emph{asymptotically
almost surely} nonsingular: that is, $\lim_{n\to\infty}\Pr(M\text{ is singular})=0$.
Koml\'os' theorem can be generalised to sparse random Bernoulli matrices
as follows.
\begin{thm}
\label{thm:ber}Fix  $\varepsilon>0$, and let $p=p(n)$
be any function of $n$ satisfying $\min(p,1-p)\geq (1+\varepsilon)\log n/n$.
Then for a random $n\times n$ random matrix $M$ with i.i.d.\ $\Ber(p)$
entries, we have
\[\lim_{n\to\infty}\Pr(M\text{ is singular})=0.\]
\end{thm}

\cref{thm:ber} is best-possible, in the sense that if $\min(p,1-p) \le(1-\varepsilon)\log n/n$,
then we actually have $\lim_{n\to\infty}\Pr(M\text{ is singular})=1$
(because, for instance, $M$ is likely to have two identical columns). That
is to say, $\log n/n$ is a \emph{sharp threshold} for singularity.
It is not clear when \cref{thm:ber} first appeared in print, but strengthenings
and variations on \cref{thm:ber} have been proved by several different
authors (see for example \cite{AE14,BR18,CV08,CV10}).

Next, let $Q$ be an $n\times n$ random matrix with independent rows,
where each row is sampled uniformly from the subset of vectors in
$\{ 0,1\} ^{n}$ having exactly $d$ ones ($Q$ is said
to be a random \emph{combinatorial} matrix). The study of such matrices was initiated by Nguyen\ \cite{Ngu13}, who proved that if $d=n/2$ then $Q$ is asymptotically almost surely
nonsingular (where $n\to\infty$ along the even integers). Strengthenings of Nguyen's theorem have been proved by several authors; see for example \cite{AP20,FJLS,Jai19,JSS20,Tra20}. Recently, Aigner-Horev and Person\ \cite{AP20} conjectured an analogue of \cref{thm:ber} for sparse random combinatorial matrices, which we prove in this note.
\begin{thm}
\label{thm:comb}Fix $\varepsilon>0$, and let $d=d(n)$
be any function of $n$ satisfying $\min(d,n-d)\geq (1+\varepsilon)\log n$.
Then for a $n\times n$ random zero-one matrix $Q$ with independent rows,
where each row is chosen uniformly among the vectors with $d$ ones, we have
\[
\lim_{n\to\infty}\Pr(Q\text{ is singular})\to0.
\]
\end{thm}

Just like \cref{thm:ber}, \cref{thm:comb} is best-possible in the sense that if $\min(d,n-d) \le(1-\varepsilon)\log n$,
then $\lim_{n\to\infty}\Pr(M\text{ is singular})=1$. \cref{thm:comb} improves on a result of Aigner-Horev and Person: they proved the same fact under the assumption that $\lim_{n\to \infty} d/(n^{1/2}\log^{3/2}n)=\infty$ (assuming that $d\le n/2$). 

The structure of this note is as follows. First, in \cref{sec:general} we prove a simple and general lemma (\cref{lem:general}) which applies to any random matrix with i.i.d.\ rows. This lemma distills the essence of (a special case of) an argument due to Rudelson and Vershinyn~\cite{RV08}. Essentially, it shows that in order to prove \cref{thm:ber} and \cref{thm:comb}, one just needs to prove some relatively crude estimates about the typical structure of the vectors in the left and right kernels of our random matrices.

Then, in \cref{sec:ber} and \cref{sec:comb} we show how to use \cref{lem:general} to give simple proofs of \cref{thm:ber} and \cref{thm:comb}. Of course, \cref{thm:ber} is not new, but its proof is extremely simple and it serves as a warm-up for \cref{thm:comb}. It turns out that in order to analyse the typical structure of the vectors in the left and right kernel, we can work over $\ZZ_q$ for some small integer $q$ (in fact, we can mostly work over $\ZZ_2$). This idea is not new (see for example, \cite{AP20,CMMM19,Fer20,FJ19,FJLS,Hua18,Mes20,NW18a,NW18b}), but the details here are much simpler.

We remark that with a bit more work, the methods in our proofs can also likely be used to prove the conclusions of \cref{thm:ber} and \cref{thm:comb} under the weaker (and strictly best-possible) assumptions that $\lim_{n\to \infty}(\min(pn,n-pn)-\log n)=\infty$ and $\lim_{n\to \infty}(\min(d,n-d)-\log n)=\infty$. However, in this note we wish to emphasise the simple ideas in our proofs and do not pursue this direction.

\textit{Notation.} All logarithms are to base $e$. We use common asymptotic notation, as follows. For real-valued functions $f(n)$ and $g(n)$, we write $f=O(g)$ to mean that there is some constant $C>0$ such that $\vert f\vert \leq Cg$. If $g$ is nonnegative, we write $f=\Omega(g)$ to mean that there is $c>0$ such that $f \geq cg$ for sufficiently large $n$. We write $f=o(g)$ to mean that $f(n)/g(n)\to 0$ as $n\to\infty$.

\textit{Acknowledgements.} We would like to thank Elad Aigner-Horev, Yury Person, and the anonymous referee, for helpful comments and suggestions.

\section{A general lemma}
\label{sec:general}

In this section we prove a (very simple) lemma which will give us
a proof scheme for both \cref{thm:ber} and \cref{thm:comb}. For a
vector $x$, let $\supp(x)$ (the \emph{support} of $x$)
be the set of indices $i$ such that $x_{i}\ne0$.
\begin{lem}
\label{lem:general}Let $\FF$ be a field, and let $A\in\FF^{n\times n}$
be a random matrix with i.i.d.\ rows $R_{1},\dots,R_{n}$. Let $\mathcal{P}\subseteq\FF^{n}$
be any property of vectors in $\FF^{n}$. Then for any $t\in\RR$,
the probability that $A$ is singular is upper-bounded by
\begin{align}
 & \Pr(x^{T}A=0\text{ for some nonzero }x\in\FF^{n}\text{ with }|\supp(x)|<t)\label{eq:small-supp}\\
 & \qquad+\frac{n}{t}\Pr(\text{there is nonzero }x\notin\mathcal{P}\text{ such that }x\cdot R_{i}=0\text{ for all }i=1,\dots,n-1)\label{eq:P}\\
 & \qquad+\frac{n}{t}\sup_{x\in\mathcal{P}}\Pr(x\cdot R_{n}=0)\label{eq:LO}
\end{align}
\end{lem}

\begin{proof}
Note that $A$ is singular if and only if there is a nonzero $x\in\FF^{n}$
satisfying $x^{T}A=0$. Let $\mathcal{E}_{i}$ be the event that $R_{i}\in\spn\{ R_{1},\dots,R_{i-1},R_{i+1},\dots,R_{n}\} $,
and let $X$ be the number of $i$ for which $\mathcal{E}_{i}$ holds.
Then by Markov's inequality and the assumption that the rows $R_{1},\dots,R_{n}$ are i.i.d., we have
\[
\Pr\left(x^{T}M=0\text{ for some }x\text{ with }|\supp\left(x\right)|\ge t\right)\le\Pr(X\ge t)\le\frac{\E X}{t}=\frac{n}{t}\Pr(\mathcal{E}_{n}).
\]
It now suffices to show that $\frac{n}{t}\Pr(\mathcal{E}_{n})$ is upper-bounded by the sum of the terms \cref{eq:P} and \cref{eq:LO}. Note that we can always choose a nonzero vector $x\in \FF^n$ with $x\cdot R_{i}=0$ for $i=1,\dots,n-1$. We interpret $x$ as a random vector depending on $R_1,\dots,R_{n-1}$ (but not $R_n$). If the event $\mathcal{E}_{n}$ occurs, we must have $x\cdot R_n=0$, so
$$\frac n t\Pr(\mathcal{E}_{n})\le \frac n t\Pr(x\notin \mathcal P)+\frac n t\Pr(x\cdot R_n=0\,|\,x\in \mathcal P).$$
Then $\frac n t\Pr(x\notin \mathcal P)$ is upper-bounded by the expression in \cref{eq:P}, and, since $x$ and $R_n$ are independent, $\frac n t\Pr(x\cdot R_n=0\,|\,x\in \mathcal P)$ is upper-bounded by the expression in \cref{eq:LO}.
\end{proof}

\section{Singularity of sparse Bernoulli matrices: a simple
proof}
\label{sec:ber}

Let us fix $0<\eps<1$. We will take $t=cn$ for some small constant $c$ (depending on $\varepsilon$), and let $\mathcal{P}$
be the property $\{ x\in\QQ^{n}:|\supp (x)|\geq t\} $.
All we need to do is to show that the three terms \cref{eq:small-supp},
\cref{eq:P} and \cref{eq:LO} in \cref{lem:general} are each of the form $o(1)$. The following
lemma is the main part of the proof.

\begin{lem}
\label{lem:ber-normal}Let $R_{1},\dots,R_{n-1}$ be the first $n-1$
rows of a random $\Ber(p)$ matrix, with $\min(p,1-p)\geq (1+\varepsilon)\log n/n$.
There is $c>0$ (depending only on $\varepsilon$) such that with probability
$1-o(1)$, no nonzero vector $x\in\QQ^{n}$ with $|\supp(x)|<cn$
satisfies $R_{i}\cdot x=0$ for all $i=1,\dots,n-1$.
\end{lem}

\begin{proof}
If such a vector $x$ were to exist, we would be able to multiply by an integer and then
divide by a power of two to obtain a vector $v\in\ZZ^{n}$ with
at least one odd entry also satisfying $|\supp(v)|<cn$
and $R_{i}\cdot v=0$ for $i=1,\dots,n-1$. Interpreting $v$ as a vector in $\ZZ_{2}^n$, we would have $R_{i}\cdot v\equiv 0 \pmod{2}$ for $i=1,\dots,n-1$ and furthermore $v\in \ZZ_{2}^n$ would be a nonzero vector consisting of less than $cn$ ones. We show that such a vector $v$ is unlikely to exist (working over $\ZZ_{2}$ discretises the problem, so that we may use a union bound).

Let $p^*=\min(p,1-p)\geq (1+\varepsilon)\log n/n$. Consider any $v\in\{ 0,1\} ^{n}$ with $|\supp(v)|=s$.
Then $R_{i}\cdot v$ for $i=1,\dots,n-1$ are i.i.d.\ $\Bin(s,p)$ random variables. Let $P_{s,p}$ denote the probability that a $\Bin(s,p)$ random variable is even. We observe
\begin{align*}
P_{s,p} & =\frac{1}{2}\left(\sum_{i=0}^{s}\binom{s}{i}p^{i}(1-p)^{s-i}+\sum_{i=0}^{s}\binom{s}{i}(-1)^{i}p^{i}(1-p)^{s-i}\right)\\
 & =\frac12+\frac{(1-2p)^{s}}2\le \frac12+\frac{(1-2p^*)^{s}}2.
\end{align*}

Then, using the fact that $e^{-t}=1-t+O(t^2)$ for $t=o(1)$, we deduce
\begin{align*}
P_{s,p}\leq \begin{cases}
e^{-\left(1+o(1)\right)sp^*} & \text{if }sp^*=o(1),\\
e^{-\Omega(1)} & \text{if }sp^*=\Omega(1).
\end{cases}
\end{align*}

Taking $r=\delta/p^*$ for sufficiently small $\delta$ (relative to
$\varepsilon$), and recalling that $p^*\geq (1+\varepsilon)\log n/n$, the probability that there exists nonzero $v\in\ZZ_{2}^n$
with $|\supp(v)|<cn$ and $R_{i}\cdot v\equiv 0 \pmod{2}$
for all $i=1,\dots,n-1$ is at most
\begin{align*}
\sum_{s=1}^{cn}\binom{n}{s}P_{s,p}^{n-1} & \le\sum_{s=1}^{r}e^{s\log n-(1-\varepsilon/3)snp^*}+\sum_{s=r+1}^{cn}e^{s(\log(n/s)+1)-\Omega(n)}\\
 & \le\sum_{s=1}^{\infty}n^{-s\varepsilon/3}+\sum_{s=1}^{cn}e^{n\left((s/n)(\log(n/s)+1)-\Omega(1)\right)}=o(1),
\end{align*}
provided $c$ is sufficiently small (relative to $\delta$).
\end{proof}
Taking $c$ as in \cref{lem:ber-normal}, we immediately see that the
term \cref{eq:P} is of the form $o\left(1\right)$. Observing that
the rows and columns of $M$ have the same distribution, and that
the event $x^{T}M=0$ is simply the event that $x\cdot C_{i}=0$ for
each column $C_{i}$ of $M$, it also follows from \cref{lem:ber-normal}
that the term \cref{eq:small-supp} is of the form $o\left(1\right)$.
Finally, the following straightforward generalisation of the well-known Erd\H os--Littlewood--Offord
theorem shows that the
term \cref{eq:LO} is of the form $o\left(1\right)$, which completes the proof of \cref{thm:ber}. This 
lemma is the only nontrivial ingredient in the proof of \cref{thm:ber}. It appears as \cite[Lemma~8.2]{CV08}, but it can also be quite straightforwardly deduced from the Erd\H os--Littlewood--Offord theorem itself.
\begin{lem}
\label{lem:L-O}Consider a (non-random) vector $x=(x_{1},\dots,x_{n})\in\RR^{n}$,
and let $\xi_{1},\dots,\xi_{n}$ be i.i.d.\ $\Ber(p)$
random variables, and let $p^*=\min(p,1-p)$. Then
\[
\max_{a\in \RR}\Pr(x_{1}\xi_{1}+\dots+x_{n}\xi_{n}=a)=O\left(\frac{1}{\sqrt{|\supp(x)|p^*}}\right).
\]
\end{lem}

\section{Singularity of sparse combinatorial matrices}
\label{sec:comb}

Let us again fix $0<\eps<1$. The proof of \cref{thm:comb} proceeds in almost exactly the same way
as the proof of \cref{thm:ber}, but there are three significant complications.
First, since the entries are no longer independent, the calculations
become somewhat more technical. Second, the rows and columns of $Q$ have different distributions,
so we need two versions of \cref{lem:ber-normal}: one for vectors in the left kernel and one for vectors in the right kernel. Third, the fact that
each row has exactly $d$ ones means that we are not quite as free
to do computations over $\ZZ_{2}$ (for example, if $d$ is even and
$v$ is the all-ones vector then we always have $Qv=0$ over $\ZZ_{2}$). For certain parts of the argument we will instead work over $\ZZ_{d-1}$.

Before we start the proof, the following lemma will allow us to restrict our attention to the case where $d\le n/2$, which will be convenient.

\begin{lem}
Let $Q\in \RR^{n\times n}$ be a matrix whose every row has sum $d$, for some $d\notin\{0,n\}$. Let $J$ be the $n\times n$ all-ones matrix. Then $Q$ is singular if and only if $J-Q$ is singular.
\end{lem}
\begin{proof}
Note that the all-ones vector $\one$ is in the column space of $Q$ (since the sum of all columns of $Q$ equals $d\one$). Hence every column of $J-Q$ is in the column space of $Q$. Therefore, if $Q$ is singular, then $J-Q$ is singular as well. The opposite implication can be proved the same way.
\end{proof}

In the rest of the section we prove \cref{thm:comb} under the assumption that $(1+\eps)\log n\le d\le n/2$ (note that if $Q$ is a uniformly random zero-one matrix with every row having exactly $d$ ones, then $J-Q$ is a uniformly random zero-one matrix with every row having exactly $n-d$ ones).

The first ingredient we will need is an analogue of \cref{lem:L-O} for ``combinatorial''
random vectors. In addition to the notion of the support
of a vector, we define a \emph{fibre} of a vector to be a set of all
indices whose entries are equal to a particular value.
\begin{lem}
\label{lem:comb-LO}Let $0\le d\le n/2$, and consider a (non-random) vector $x\in\RR^{n}$ whose largest fibre
has size $n-s$, and let $\gamma\in\{ 0,1\} ^{n}$ be a random
zero-one vector with exactly $d$ ones. Then
\[
\max_{a\in \RR}\Pr(x\cdot\gamma=a)=O\left(\sqrt{n/(sd)}\right).
\]
\end{lem}

We deduce \cref{lem:comb-LO} from
the $p=1/2$ case of \cref{lem:L-O} (that is, from the Erd\H os--Littlewood--Offord
theorem~\cite{Erd45}).

\begin{proof}
The case $p=1/2$ is treated in \cite[Proposition~4.10]{LLTTY17}; this proof proceeds along similar lines. Let $p=d/n\leq 1/2$. We realise the distribution
of $\gamma$ as follows. First choose $d=pn$ random disjoint pairs
$\left(i_{1},j_{1}\right),\dots,\left(i_{pn},j_{pn}\right)\in\left\{ 1,\dots,n\right\} ^{2}$ (each having distinct entries),
and then determine the 1-entries in $\gamma$ by randomly choosing one
element from each pair.

We first claim that with probability $1-e^{-\Omega\left(sp\right)}$,
at least $\Omega\left(sp\right)$ of our pairs $\left(i,j\right)$
have $x_{i}\ne x_{j}$ (we say such a pair is \emph{good}). To see
this, let $I$ be a union of fibres of $x$, chosen such that $\left|I\right|\ge n/3$
and $n-\left|I\right|\geq s/3$ (if $s\le 2n/3$ we can
simply take $I$ to be the largest fibre of $x$, and otherwise we
can greedily add fibres to $I$ until $\left|I\right|\ge n/3$). To prove our claim, we will prove that in fact with the desired probability there are $\Omega(sp)$ different $\ell$ for which $i_\ell\notin I$ and $j_\ell\in I$.

Let $f=\ceil{pn/6}$ and let $S$ be the set of $\ell\le f$ for which
$i_{\ell}\notin I$. So, $\left|S\right|$ has a hypergeometric distribution
with mean $(n-|I|)f/n=\Omega\left(sp\right)$, and by a Chernoff bound (see for
example \cite[Theorem~2.10]{JLR00}), we have $\left|S\right|=\Omega\left(sp\right)$ with
probability $1-e^{-\Omega\left(sp\right)}$. Condition on such an outcome
of $i_{1},\dots,i_{f}$. Next, let $T$ be the set of $\ell\in S$ for which $j_{\ell}\in I$.
Then, conditionally, $\left|T\right|$ has a hypergeometric distribution
with mean at least $(|I|-f)|S|/n=\Omega\left(sp\right)$, so again using a Chernoff
bound we have $\left|T\right|=\Omega\left(sp\right)$ with probability
$1-e^{-\Omega\left(sp\right)}$, as claimed.

Now, condition on an outcome of our random pairs such that at least
$\Omega(sp)$ of them are good. Let $\xi_{\ell}$ be the indicator random variable for the event that
$i_{\ell}$ is chosen from the pair $\left(i_{\ell},j_{\ell}\right)$,
so $\xi_{1},\dots,\xi_{pn}$ are i.i.d.\ $\Ber\left(1/2\right)$
random variables, and $x\cdot\gamma=a$ if and only if 
\[
(x_{i_{1}}-x_{j_{1}})\xi_{1}+\dots+(x_{i_{pn}}-x_{j_{pn}})\xi_{1}=a-x_{j_{1}}-\dots-x_{j_{pn}}.
\]
Under our conditioning, $\Omega\left(sp\right)$ of the $x_{i_{\ell}}-x_{j_{\ell}}$
are nonzero, so by \cref{lem:L-O} with $p=1/2$, conditionally we have $\Pr\left(x\cdot\gamma=a\right)\le O\left(1/\sqrt{sp}\right)$.
We deduce that unconditionally
\[
\Pr(x\cdot\gamma=0)\le e^{-(sp)}+O(1/\sqrt{sp})=O(1/\sqrt{sp})=O(\sqrt{n/(sd)}),
\]
as desired.
\end{proof}

The proof of \cref{thm:comb} then reduces to the following two lemmas. Indeed, for a constant $c>0$ (depending on $\varepsilon$) satisfying the statements in \cref{lem:comb-normal-hard,lem:comb-normal}, we can take $t=cn/\log d$, and 
\[\mathcal P=\{x\in \QQ^n:x\text{ has largest fibre of size at most }(1-c/\log d)n\}.\]
We can then apply \cref{lem:general}. By \cref{lem:comb-normal-hard}, the term \cref{eq:small-supp} is bounded by $o(1)$, by \cref{lem:comb-normal} the term \cref{eq:P} is bounded by $(n/t)\cdot n^{-\Omega(1)}=(\log d/c)\cdot n^{-\Omega(1)}=o(1)$, and by \cref{lem:comb-LO} the term \cref{eq:LO} is bounded by $(n/t)\cdot O\left(\sqrt{n\log d/(cnd)}\right)= O(\log^{3/2}d/\sqrt d)=o(1)$.

\begin{lem}
\label{lem:comb-normal-hard}Let $Q$ be a random combinatorial matrix (with $d$ ones in each row),
with $(1+\varepsilon)\log n\le d\le n/2$. There
is $c>0$ (depending only on $\varepsilon$) such that with probability
$1-o(1)$, there is no nonzero vector $x\in\QQ^{n}$ with
$|\supp(x)|<cn/\log d$ and $x^{T}Q=0$.
\end{lem}

\begin{lem}
\label{lem:comb-normal}Let $R_{1},\dots,R_{n-1}$ be the first $n-1$ rows of a random combinatorial matrix (with $d$ ones in each row), with $(1+\varepsilon)\log n\le d\le n/2$. There is $c>0$ (depending only on $\varepsilon$)
such that with probability $1-n^{-\Omega(1)}$, every nonzero $x\in \QQ^n$ satisfying $R_{i}\cdot x=0$
for all $i=1,\dots,n-1$ has largest fibre of size at most $(1-c/\log d)n$.
\end{lem}

\begin{proof}[Proof of \cref{lem:comb-normal-hard}]
As in \cref{lem:ber-normal}, it suffices to work over $\ZZ_{2}$.
Let $C_{1},\dots,C_{n}$ be the columns of $Q$, consider any $v\in\ZZ_{2}^{n}$
with $|\supp(v)|=s$, and let $\mathcal{E}_v$
be the event that $C_{i}\cdot v\equiv 0\pmod{2}$ for $i=1,\dots,n$. Note that $\mathcal{E}_v$
only depends on the submatrix $Q_v$ of $Q$ containing only those
rows $j$ with $v_{j}=1$ (and $\mathcal{E}_v$ is precisely the event that every column of $Q_v$ has an even sum).

Let $p=d/n\le 1/2$, let $M_v$ be a random $s\times n$ matrix with i.i.d.\ $\Ber(p)$
entries, and let $\mathcal{E}_v'$
be the event that every column in $M_v$ has an even sum. Note that $M_v$
is very similar to $Q_v$, so the probability of $\mathcal E_v$ is very similar to the probability of $\mathcal E_v'$. Indeed, writing $R_{1},\dots,R_{s}$ and $R_{1}',\dots,R_{s}'$
for the rows of $Q_v$ and $M_v$ respectively, and writing $s_j=|\supp(R_j')|$, for each $j$ we have $s_j\sim \Bin(n,p)$, so an elementary computation using Stirling's formula shows that $\Pr(s_j=d)=\Omega(1/\sqrt{d})=e^{-O(\log d)}$. Hence
\[
\Pr(\mathcal{E}_v)=\Pr(\mathcal{E}_v'\,|\,s_j=d\text{ for all }j)\le\Pr(\mathcal{E}_v')/\Pr(s_j=d\text{ for all }j)=e^{O(s\log d )}\Pr(\mathcal{E}_v')=e^{O(s\log (pn))}\Pr(\mathcal{E}_v').
\]
Recalling the quantity $P_{s,p}$ from the proof of \cref{lem:ber-normal},
we have
\[
\Pr(\mathcal{E}_v')=P_{s,p}^{n}=\begin{cases}
e^{-\left(1+o(1)\right)spn} & \text{if }sp=o(1),\\
e^{-\Omega(n)} & \text{if }sp=\Omega(1),
\end{cases}
\]
so if $s\le cn/\log d=cn/\log(pn)$ for small $c>0$, then we also have
\[
\Pr(\mathcal{E}_v)\leq \begin{cases}
e^{-\left(1+o(1)\right)spn} & \text{if }sp=o(1),\\
e^{-\Omega(n)} & \text{if }sp=\Omega(1).
\end{cases}
\]
Let $P_s=\Pr(\mathcal{E}_v)$ (which only depends on $s$). We can now conclude the proof in exactly the same way as in \cref{lem:ber-normal}. Taking $r=\delta/p$ for sufficiently small $\delta$ (relative to
$\varepsilon$), the probability that there exists nonzero $v\in\ZZ_{2}^n$
with $|\supp(v)|<cn/\log d$ and $C_{i}\cdot v\equiv 0\pmod{2}$
for all $i=1,\dots,n$ is at most 
\begin{align*}
\sum_{s=1}^{cn/\log d}\binom{n}{s}P_{s} & \le\sum_{s=1}^{r}e^{s\log n-(1-\varepsilon/3)snp}+\sum_{s=r+1}^{cn/\log d}e^{s(\log(n/s)+1)-\Omega(n)}\\
 & \le\sum_{s=1}^{\infty}n^{-s\varepsilon/3}+\sum_{s=1}^{cn/\log d}e^{n\left((s/n)(\log(n/s)+1)-\Omega(1)\right)}=o(1),
\end{align*}
provided $c$ is sufficiently small (relative to $\delta$).
\end{proof}

We will deduce \cref{lem:comb-normal} from the following lemma.

\begin{lem}
\label{lem:hypergeometric}Suppose $p\le1/2$ and $pn\to\infty$, and let $\gamma\in\{ 0,1\} ^{n}$ be
a random vector with exactly $pn$ ones. Let $q\geq 2$ be an integer and consider a (non-random) vector $v\in\ZZ_{q}^{n}$
whose largest fibre has size $n-s$. Then for any $a\in \ZZ_q$ we have $\Pr(v\cdot\gamma\equiv a \pmod{q})\le P_{p,n,s}$
for some $P_{p,n,s}$ (only depending on $p$, $n$ and $s$) satisfying
\[
P_{p,n,s}=\begin{cases}
e^{-\Omega(1)} & \text{when }sp=\Omega(1),\\
e^{-\left(1-o(1)\right)sp} & \text{when }sp=o(1)
\end{cases}
\]
\end{lem}

\begin{proof}
As in the proof of \cref{lem:comb-LO}, we realise the distribution of $\gamma$ by first choosing $pn$ random disjoint pairs $(i_{1},j_{1}),\dots,(i_{pn},j_{pn})\in\{ 1,\dots,n\} ^{2}$,
and then randomly choosing one element from each pair to comprise the 1-entries of $\gamma$.

Let $\mathcal{E}$ be the event that $v_{i}\ne v_{j}$ for at least
one of our random pairs $(i,j)$. Then $\Pr(v\cdot\gamma\equiv a \pmod {q}\,|\,\mathcal{E})\le1/2$, and therefore $\Pr(v\cdot \gamma \equiv a \pmod q ) \leq 1 - \Pr(\mathcal E)/2$.
So, it actually suffices to prove that
\[
\Pr(\mathcal{E})\ge\begin{cases}
\Omega(1) & \text{when }sp=\Omega(1),\\
\left(2-o(1)\right)sp & \text{when }sp=o(1).
\end{cases}
\]
If $s\ge n/3$ (this can only occur if $sp=\Omega(1)$),
then we can choose $J\su \{1,\dots,n\}$ to be a union of fibres of the vector $v\in \ZZ_q^n$ such that $n/3\le|J|\le2n/3$.
In this case, 
\[
\Pr(\mathcal{E})\ge\Pr(i_{1}\in J,\,j_{1}\notin J)=\Omega(1),
\]
as desired. So, we assume $s<n/3$, and let $I\su \{1,\dots,n\}$ be the set of indices in the largest fibre of $v$ (so $|I|=n-s$).
Note that $\mathcal{E}$ occurs whenever there is a pair $\{ i_{k},j_{k}\} $
with exactly one element in $I$.

Let $\mathcal{F}$ be the event that $i_{k}\in I$ for all $k=1,\dots,pn$. We have
\[
\Pr(\mathcal{E}\,|\,\mathcal{F})\ge1-(1-s/n)^{pn}=\begin{cases}
\Omega(1) & \text{when }sp=\Omega(1),\\
\left(1-o(1)\right)sp & \text{when }sp=o(1),
\end{cases}
\]
and
\[
\Pr(\mathcal{E}\,|\,\overline{\mathcal{F}})\ge(n-s-pn)/(n-pn)=\begin{cases}
\Omega(1) & \text{when }sp=\Omega(1),\\
1-o(1) & \text{when }sp=o(1).
\end{cases}
\]
This already implies that if $sp=\Omega(1)$, then $\Pr(\mathcal{E})=\Omega(1)$
as desired. If $sp=o(1)$ then $\Pr(\mathcal{F})\le(1-s/n)^{pn}=1-\left(1+o(1)\right)sp$,
so
\[
\Pr(\mathcal{E})=\Pr(\mathcal{F})\Pr(\mathcal{E}\,|\,\mathcal{F})+\Pr(\overline{\mathcal{F}})\Pr(\mathcal{E}\,|\,\mathcal{\overline{\mathcal{F}}})\geq \left(2-o(1)\right)sp,
\]
as desired.
\end{proof}
\begin{proof}[Proof of \cref{lem:comb-normal}]
Let $q=d-1$. It suffices to prove that with probability $1-o(1)$ there is no nonconstant ``bad'' vector
$v\in\ZZ_{q}^n$ whose largest fibre has size at least $(1-c/\log q)n$
and which satisfies $R_{i}\cdot v\equiv 0\pmod{q}$ for all $i=1,\dots,n-1$. (Note that by the choice of $q$, if $v\in \ZZ_{q}^n$ is constant and nonzero, then it is impossible to have $v\cdot R_1=0$).

Let $p=d/n$, consider any $v\in\ZZ_{q}^n$ whose largest fibre has size $n-s$, and
consider any $i\in \{1,\dots,n-1\}$. Then $R_{i}\cdot v$ is of the form in \cref{lem:hypergeometric},
so taking $r=\delta/p$ for sufficiently small $\delta$ (relative to $\varepsilon$), the probability
that such a bad vector exists is at most 
\begin{align*}
\sum_{s=1}^{c'n/\log q}\binom{n}{s}q^{s+1}P_{p,n,s}^{n-1} & \le\sum_{s=1}^{r}e^{s\log n+(s+1)2\sqrt{pn}-(1-\varepsilon/3)spn}+\sum_{s=r+1}^{c'n/\log q}e^{s(\log(n/s)+1)+cn+2\sqrt{pn}-\Omega(n)}\\
 & \le\sum_{s=1}^{\infty}n^{-s\varepsilon/3}+\sum_{s=1}^{c'n/\log q}e^{n\left((s/n)(\log(n/s)+1)-\Omega(1)\right)}=n^{-\Omega(1)},
\end{align*}
provided $c'>0$ is sufficiently small (relative to $\delta$) and $n$ is sufficiently large.
\end{proof}


\end{document}